\documentclass{rmmcart}

\usepackage{amssymb,amsmath,amsthm,graphicx,float,epsf}

\newcommand{\beq}[1]{ \begin{equation}\label{#1} }
\newcommand{\eeq}{\end{equation}}
\numberwithin{equation}{section}
\DeclareMathOperator*{\esssup}{ess\,sup}
\restylefloat{figure}

\def\RR{{\mathbb R}}

\usepackage{fullpage}
\newtheorem{guess}{Theorem}
\newtheorem{uess}{Lemma}
\newtheorem{corollary}{Corollary}
\newtheorem{definition}{Definition}
\newtheorem{example}{Example}
\newtheorem{remark}{Remark}
\newtheorem{proposition}{Proposition}

\title{Explicit stability tests for linear neutral  delay equations 
using infinite series}

\author{Leonid Berezansky} 
\address{Department of Mathematics,
Ben-Gurion University of the Negev, 
Beer-Sheva 84105, Israel}
\email{brznsky@math.bgu.ac.il}

\author{Elena Braverman}
\address{Department of Mathematics and Statistics,
University of Calgary, 2500 University Drive N.W., Calgary, AB T2N 1N4, Canada}
\email{maelena@ucalgary.ca}
\thanks{Partially supported by the NSERC research grant RGPIN-2015-05976.}
\date{July 5, 2018; revised September 6, 2018}

\keywords{Neutral  equations in the Hale form,  
uniform exponential stability, Bohl-Perron theorem, variable delays,
explicit stability conditions}

\subjclass{34K40, 34K20, 34K06}

\begin{document}

\begin{abstract}
We obtain new explicit exponential stability conditions
for the linear scalar neutral equation with two bounded delays
$
(x(t)-a(t)x(g(t)))'+b(t)x(h(t))=0,
$
where $|a(t)| \leq A_0 < 1$, $0<b_0\leq b(t)\leq B_0$, 
assuming that all parameters of the
equation are measurable functions. 

To analyze exponential stability, we apply  the Bohl-Perron theorem and a reduction of a neutral equation to
an equation with an infinite number of non-neutral delay terms. This method has never been used before for this neutral equation; its application allowed to omit a usual restriction $|a(t)|<\frac{1}{2}$ in known asymptotic stability tests
and consider variable delays. 
\end{abstract}

\maketitle

\section{Introduction and Preliminaries}

Neutral differential equations have many applications in control theory, ecology, biology, physics,
see, for example, \cite{Fridman,Gop,H,KolmMysh,Kuang}. 
The aim of the present paper is to 
obtain new explicit exponential stability conditions for the equation
\begin{equation}
\label{1} 
(x(t)-a(t)x(g(t)))'=-b(t)x(h(t)).
\end{equation}

Its particular case with constant delays
\begin{equation}\label{2}
\left( x(t)-a(t)x(t-\sigma) \right)'+ b(t)x(t-\tau)=0,
\end{equation}
where $\tau,\sigma>0$, $a,b \in C([t_0,\infty), \RR)$, $b(t)\geq 0$,
has been extensively investigated \cite{TangZou,Yu}.

\begin{proposition}\label{proposition2} \cite{Yu}
If $\displaystyle \int_{t_0}^{\infty} b(s)ds=+\infty$, $|a(t)|\leq A_0 < 1$ and
$\displaystyle  \limsup_{t\rightarrow\infty} \int_{t-\tau} ^t b(s)ds<\frac{3}{2}-2A_0(2-A_0)$
then equation (\ref{2}) is asymptotically stable.
\end{proposition}

In the non-neutral case $a(t)\equiv 0$,  Proposition \ref{proposition2}
turns into the sharp stability result with the constant $\frac{3}{2}$.
There are several improvements and extensions of Proposition~\ref{proposition2}, 
we cite the test obtained in \cite{TangZou}.

\begin{proposition}\label{proposition2a} \cite{TangZou}
Let $b(t) \geq 0$,  $\displaystyle \int_{t_0}^{\infty} b(s)ds=+\infty$, $|a(t)|\leq A_0 < 1$,
and at least one of the following conditions hold:
\vspace{2mm}

a) $\displaystyle A_0<\frac{1}{4},~~\limsup_{t\rightarrow\infty} \int_{t-\tau}^t b(s)ds<\frac{3}{2}-2A_0$; 
\vspace{2mm}

b) $\displaystyle \frac{1}{4}\leq A_0 <\frac{1}{2},~~\limsup_{t\rightarrow\infty} \int_{t-\tau}^t b(s)ds<
\sqrt{2 \left( 1-2A_0 \right) }$.
\vspace{2mm}

Then equation (\ref{2}) is asymptotically stable.
\end{proposition}

It is easy to see that both propositions assume $|a(t)|<\frac{1}{2}$.
Additional restrictions are that the delays are constant, coefficients $a(t)$ and $b(t)$ are continuous functions. 
In the present paper, we omit some of these restrictions.

To analyze exponential stability, we apply  the Bohl-Perron theorem and a reduction of a neutral equation to
an equation with an infinite number of non-neutral delay terms. This method has never been used before for  neutral equation
\eqref{1}.



We consider (\ref{1}) under a number of the following assumptions:
\\
(a1) $a, b, g, h$ are Lebesgue measurable on $[0,+\infty)$, and there exist positive constants $A_0,b_0,B_0$ such that  
$\left|  a(t) \right| \leq A_0<1$, $0< b_0\leq b(t)\leq B_0$; \\
(a2)  mes~$\Omega=0\Longrightarrow$ mes~$g^{-1}(\Omega)=0$, 
where mes$~\Omega$ is  the Lebesgue measure of the set $\Omega$;\\
(a3) $g(t)\leq t$, $h(t)\leq t$, $\limsup_{t\rightarrow\infty} g(t)=\infty$, $\limsup_{t\rightarrow\infty} h(t)=\infty$;\\
(a4) $0\leq t-g(t)\leq \sigma$, $0\leq \delta\leq t-h(t)\leq \tau$, $t\geq t_0$ for some $t_0\geq 0$ and $t\geq t_0$.

Together with  (\ref{1}) we consider for each $t_0 \geq 0$ an initial value problem
\begin{equation}
\label{3}
(x(t)-a(t)x(g(t)))'+b(t)x(h(t))=f(t), ~~t\geq t_0,~~
x(t)=\varphi(t), ~ t \leq t_0
\end{equation}
and assume that for $f$ and $\varphi$  the following condition holds:
\\
(a5) $f:[t_0,+\infty)\rightarrow {\mathbb R}$ is Lebesgue measurable locally essentially
bounded, 
$\varphi :(-\infty,t_0)\rightarrow {\mathbb R}$ is a Borel
measurable bounded function.

\begin{definition} 
A function $x: {\mathbb R} \rightarrow {\mathbb R}$ is called {\bf a solution of problem}  (\ref{3}) if 
the difference $x(t)-a(t)x(g(t))$ is absolutely continuous on each interval $[t_0,c]$,
$x$ satisfies the  equation in (\ref{3}) for almost all $t\in [t_0,+\infty)$ and the initial condition
in (\ref{3}) for $t\leq t_0$.
\end{definition}

There exists a unique solution of problem (\ref{3}) if conditions (a1)-(a3),(a5) hold, see \cite{Gil,H}.

Consider the initial value problem for the equation with one non-neutral delay term
\begin{equation}
\label{5}
x^{\prime}(t)+b(t)x(h(t))=f(t),~ t \geq t_0, ~x(t)=0,~t\leq t_0,
\end{equation}
where $b(t), f(t)$ and  $h(t) \leq t$ are Lebesgue measurable locally essentially bounded functions.

\begin{definition} 
For each $s\geq t_0$ the solution $X(t,s)$ of the problem
\begin{equation*}
x^{\prime}(t)+b(t)x(h(t))=0,~ t \geq t_0, ~x(t)=0,~t<s,~x(s)=1
\end{equation*}
is called {\bf a fundamental function of equation}  (\ref{5}).
We assume $X(t,s)=0$ for  $0\leq t<s$.
\end{definition}

\begin{uess}  \cite{AzbSim}
\label{lemma2}
Let (a1)-(a3),(a5) hold. 
The solution of
problem  (\ref{5})  can be presented in the form
\begin{equation*}
x(t)=\int_{t_0}^t X(t,s)f(s)ds.
\end{equation*}
\end{uess}

\begin{definition} 
Equation (\ref{1}) is {\bf (uniformly) exponentially stable} 
if there exist positive numbers $M$ and $\gamma$ such that 
the solution of problem  (\ref{3})
with $f \equiv 0$ has the estimate 
\begin{equation*}
|x(t)|\leq M e^{-\gamma (t-t_0)} \sup_{t \in (-\infty,  t_0]}|\varphi(t)|,~~t\geq t_0,
\end{equation*}
where $M$ and $\gamma$ do not depend on $t_0 \geq 0$ and $\varphi$.
\end{definition}


Next, we present the Bohl-Perron theorem.

\begin{uess}\label{lemma3}\cite[Theorem 6.1]{Gil}
Let (a1)-(a2),(a4) and (a5) hold.
Assume that the solution of the problem 
\begin{equation}\label{10}
(x(t)-a(t)x(g(t)))'+b(t)x(h(t))=f(t), ~x(t)=0,~t\leq t_0
\end{equation}
is essentially bounded on $[t_0,+\infty)$ for any 
essentially bounded $f: [t_0,+\infty) \to {\mathbb R}$. 
Then equation (\ref{1}) is uniformly exponentially stable.
\end{uess}

\begin{remark}
\label{remark1a}
In Lemma~\ref{lemma3} we can consider
boundedness of solutions not for all essentially bounded functions $f: [t_0,+\infty) \to {\mathbb R}$  but only
for those which satisfy $f(t)=0$, $t \in [t_0,t_1)$, for any fixed $t_1>t_0$, see \cite{BB3}.
We will further apply this fact in the paper without an additional reference.
\end{remark}

Consider now a linear equation with a measurable single delay and a locally essentially bounded non-negative coefficient
\begin{equation}\label{9}
x^{\prime}(t)+b(t)x(h_0(t))=0, ~~b(t)\geq 0,~~ 0\leq t-h_0(t)\leq \tau_0.
\end{equation}
Let $X_0(t,s)$ be a fundamental function of (\ref{9}).

\begin{uess}\label{lemma4}\cite{BB3}
Assume that $X_0(t,s)>0$ , $t\geq s\geq t_0$. Then 
$$
\int_{t_0+\tau_0}^t X_0(t,s) b(s)ds\leq 1.
$$
\end{uess}

\begin{uess}\label{lemma5}\cite{BB3,GL}
If for some $t_0\geq 0$ the inequality
$\displaystyle \int_{h_0(t)}^t b(s) ds\leq \frac{1}{e}$ holds for $t\geq t_0$
then $X_0(t,s)>0$ for $t\geq s\geq t_0$.
If in addition $b(t)\geq b_0>0$ then equation (\ref{9}) is 
exponentially stable.
\end{uess}
For a fixed bounded interval $I=[t_0,t_1]$, consider the space $L_I:=L_{\infty}[t_0,t_1]$ of all essentially bounded on $I$
functions with the 
norm $|y|_I= \esssup_{t\in I} |y(t)|$, for an unbounded interval
denote $\displaystyle \|f\|_{[t_0,+\infty)}=\esssup_{t\geq t_0} |f(t)|$, $E$ is the identity operator.
Define a linear operator 
on the space 
$L_I$ as 
$$\displaystyle 
(Sy)(t)=\left\{\begin{array}{ll}
a(t)y(g(t)),& g(t)\geq t_0,\\
0,& g(t)<t_0,\\
\end{array}\right. 
$$
where $|a(t)| \leq A_0<1$ is Lebesgue measurable, $g(t) \leq t$ is measurable satisfying (a2). 

\begin{uess}\label{lemmaS} \cite{ABR}
$E-S$ is invertible in the space 
$L_I$ for any $t_1 >t_0$, 
$\displaystyle (E-S)^{-1}=\sum_{j=0}^{\infty} S^j$, where $S^0=E$, 
and the operator norm satisfies
\begin{equation}
\label{star}
\|(E-S)^{-1}\|_{L_I \to L_I} \leq \frac{1}{1-\|a\|_I} \, .
\end{equation}
\end{uess}

\section{Stability Results}


Consider initial value problem (\ref{10})
with  $\|f\|_{[t_0,+\infty)}< +\infty$. 
Further, we assume that a product equals one if it contains no factors, for example, 
$\displaystyle \prod_{k=0}^{-1}=1$.

Define
\begin{equation}
\label{tau0}
\tau_0 = \frac{1- \| a \|_{[t_0,+\infty)}}{e \| b \|_{[t_0,+\infty)}}.
\end{equation}

\begin{guess}\label{theorem1}
 Assume that (a1),(a2),(a4) hold and there is $\alpha\in [0,1]$ 
such that for $t\geq t_0$,
$a(t) \geq a_0>0$, $\alpha\tau_0\leq \delta$  and  
\begin{equation}\label{1a}
\displaystyle \tau \|b\|_{[t_0,\infty)}+\frac{\sigma \|a\|_{[t_0,\infty)} 
\|b\|_{[t_0,\infty)} (1-a_0)}{(1-\|a\|_{[t_0,\infty)})^2 }< \left( 1-\|a\|_{[t_0,\infty)}  \right) \left(1+\frac{\alpha}{e}\right).
\end{equation}

Then equation (\ref{1}) is uniformly exponentially stable.
\end{guess}

\begin{proof}
Consider initial value problem (\ref{10})
with  $\|f\|_{[t_0,+\infty)}< +\infty$, $f(t)=0$ for $t \in [t_0, t_0+\tau]$  and let $I=[t_0,t_1]$ for some $t_1>t_0+\tau$. 
By (\ref{star})
the solution  $x$ of (\ref{10}) and the derivative of the function $y(t)=x(t)-a(t)x(g(t))$ satisfy
$y'(t)=-b(t)x(h(t))+f(t)$, $x=(E-S)^{-1}y$ and
\begin{equation}
\label{star1}
|x|_I\leq \frac{1}{1-\|a\|_{[t_0,+\infty)}}|y|_I,~~
|y^{\prime}|_I\leq \frac{\|b\|_{[t_0,+\infty)}}{1-\|a\|_{[t_0,+\infty)}}|y|_I+\|f\|_{[t_0,+\infty)}.
\end{equation}
By Lemma~\ref{lemmaS}, (\ref{10})  
is equivalent to the problem for the equation with an infinite number of delays
\begin{equation}\label{11}
y'(t)=
 -  b(t)\sum_{j=0}^{\infty} \prod_{k=0}^{j-1} a\left(h( g^{[k]}(t)) \right) 
y \left( h(g^{[j]}(t) ) \right)+f(t),~~ y(t)=0,~ t\leq t_0,
\end{equation}
where 
\begin{equation}\label{11a}
g^{[0]}(t)=t, ~g^{[1]}(t)=g(t),~ g^{[k]}(t) = g\left( g^{[k-1]}(t) \right),~ ~~k\in {\mathbb N}. 
\end{equation}

Denote 
\begin{equation}
\label{Bdef}
B(t)= b(t)\sum_{j=0}^{\infty} \prod_{k=0}^{j-1} a\left( h(g^{[k]}(t)) \right),
\end{equation}
where we can assume that $a(t)=a_0$ for $t<t_0$.
Consider the delay equation
\begin{equation}\label{12}
x'(t)+B(t)x(t-\alpha\tau_0)=0,
\end{equation}
where $B$ is defined in \eqref{Bdef} and $\tau_0$ in \eqref{tau0}.

Using the bounds for $a$ and $b$, we obtain
$ \displaystyle \frac{b_0}{1-a_0}\leq B(t)\leq \frac{\|b\|_{[t_0,\infty)}}{1-\|a\|_{[t_0,\infty)}}.$

Equation (\ref{11}) can be rewritten in the form 
$$
y'(t)+B(t)y(t-\alpha \tau_0)= 
b(t)\sum_{j=0}^{\infty} \prod_{k=0}^{j-1} a\left(h( g^{[k]}(t)) \right)
\int\limits_{h(g^{[j]}(t))}^{t-\alpha \tau_0}  y'(\xi)d\xi+ f(t).
$$
Since $B(t)\geq b_0$ and $\alpha\tau_0 B(t)\leq \frac{1}{e}$, by Lemma~\ref{lemma5} 
equation (\ref{12}) is exponentially stable, and its fundamental function is positive:  $X_0(t,s)>0$, 
$t\geq s\geq t_0$.

We have 
$$
y(t)=\int\limits_{t_0}^t X_0(t,s)B(s) \left( \frac{b(s)}{B(s)}
\sum_{j=0}^{\infty} \prod_{k=0}^{j-1} a\left( h(g^{[k]}(s) )\right) 
\int\limits_{h(g^{[j]}(s))}^{s-\alpha \tau_0} 
y'(\xi)d\xi \right)ds+ f_1(t),
$$
where 
$\displaystyle \|f_1\|_{[t_0,\infty)} \leq \int_{t_0}^t X_0(t,s) |f(s)|\, ds <\infty$.
Next, the arguments satisfy 
$$
t-g(t)\leq \sigma, t-g^{[2]}(t)=t-g(t)+(g(t)-g(g(t)))\leq 2\sigma,\dots, t-g^{[n]}(t)\leq n\sigma,
$$$$
t-h(t)\leq \tau, t-h(g(t))=t-g(t)+(g(t)-h(g(t)))\leq \sigma+\tau,\dots, t-h(g^{[n]}(t))\leq n\sigma+\tau,
$$
$$
t-h(t)\geq \delta, t-h(g(t))=t-g(t)+(g(t)-h(g(t)))\geq \delta,\dots, t-h(g^{[n]}(t))\geq \delta,
$$
hence $t-h(g^{[n]}(t))\geq \delta\geq \alpha\tau_0$.
Therefore $t-\alpha \tau_0-h(g^{[n]}(t))=t-h(g^{[n]}(t))-\alpha\tau_0 \geq \delta-\alpha\tau_0\geq 0$.

Then for $t\in I$, by (\ref{Bdef}),
\begin{align*}
& \frac{b(s)}{B(s)}~
\sum_{j=0}^{\infty} \prod_{k=0}^{j-1} a\left( h(g^{[k]})(s) \right)
\int_{h(g^{[j]}(s))}^{s-\alpha\tau_0} \left| y'(\xi) \right| \,d\xi
\\
\leq & \frac{b(s)}{B(s)}
\sum_{j=0}^{\infty} \prod_{k=0}^{j-1} a\left( h(g^{[k]})(s) \right) (\tau-\alpha\tau_0+j \sigma)
~|y'|_I
\\
= &  \left[ \frac{(\tau-\alpha\tau_0) b(s)}{B(s)} \sum_{j=0}^{\infty} \prod_{k=0}^{j-1} a\left( h(g^{[k]})(s) \right)  + \frac{b(s)}{B(s)} \sum_{j=0}^{\infty} \prod_{k=0}^{j-1} a\left( h(g^{[k]})(s) \right) j \sigma \right] ~|y'|_I
\\
= &  \left[ \frac{(\tau-\alpha\tau_0) B(s)}{B(s)} + \frac{b(s)}{B(s)} \sum_{j=0}^{\infty} \prod_{k=0}^{j-1} a\left( h(g^{[k]})(s) \right) j \sigma \right] ~|y'|_I
\\
\leq & \left[ \tau-\alpha\tau_0+(1-a_0)\|a\|_{[t_0,\infty)}\sigma\sum_{j=1}^{\infty} j\|a\|_{[t_0,\infty)}^{j-1} \right] |y'|_I
\\
\leq & 
\left[ 
\tau-\alpha\tau_0+\frac{\sigma \|a\|_{[t_0,\infty)}(1-a_0)}{(1-\|a\|_{[t_0,\infty)})^2}\right]
\frac{\|b\|_{[t_0,\infty)}}{1-\|a\|_{[t_0,\infty)}}|y|_I+M_1,
\end{align*}
where the finite constant $M_1$ does not depend on $I$, and the transition between the  fourth and the fifths rows of the inequality is due to 
$$
\frac{b(s)}{B(s)} = \left( \sum_{j=0}^{\infty} \prod_{k=0}^{j-1} a \left( h(g^{[k]})(s) \right) \right)^{-1}
\leq  \left( \sum_{j=0}^{\infty} a_0^j  \right)^{-1} = 1-a_0.
$$
By Lemma~\ref{lemma4} the solution of problem (\ref{10}) satisfies 
$$\displaystyle
|y|_I \leq \left[\tau - \alpha \tau_0 +\frac{\sigma \|a\|_{[t_0,\infty)}(1-a_0)}{(1-\|a\|_{[t_0,\infty)})^2}\right]
\frac{\|b\|_{[t_0,\infty)}}{1-\|a\|_{[t_0,\infty)}}|y|_I+M_2,
$$
where the constant $M_2$ does not dependent on $I$. 
Inequality \eqref{1a} implies
$$\displaystyle
\left[\tau-\alpha\tau_0+\frac{\sigma \|a\|_{[t_0,\infty)}(1-a_0)}{(1-\|a\|_{[t_0,\infty)})^2}\right]
\frac{\|b\|_{[t_0,\infty)}}{1-\|a\|_{[t_0,\infty)}}<1.
$$

Hence $|y(t)|\leq M$ for $t \geq t_0$, for some constant $M$ which does not depend on the interval $I$.
Therefore $x$ is a bounded function 
on $[t_0,\infty)$. By Lemma \ref{lemma3}, equation (\ref{1}) is uniformly exponentially stable.
\end{proof}

Assuming $\alpha=1$ and $\alpha=0$ in Theorem~\ref{theorem1}, we get the following stability tests.

\begin{corollary}\label{main}
Assume that (a1),(a2),(a4) are satisfied, for $t\geq t_0$, $a(t) \geq a_0>0$ and at least on of the following conditions holds:

a)
$\displaystyle
\tau_0\leq \delta,~ \displaystyle \tau \|b\|_{[t_0,\infty)}+\frac{\sigma \|a\|_{[t_0,\infty)} 
\|b\|_{[t_0,\infty)} (1-a_0)}{(1-\|a\|_{[t_0,\infty)})^2 }< \left( 1-\|a\|_{[t_0,\infty)} \right) \left(1+\frac{1}{e}\right)
$
\vspace{2mm}

b)
$\displaystyle
\tau \|b\|_{[t_0,\infty)}+\frac{\sigma \|a\|_{[t_0,\infty)} \|b\|_{[t_0,\infty)} (1-a_0)}{(1-\|a\|_{[t_0,\infty)})^2 }
<1-\|a\|_{[t_0,\infty)}.
$
\vspace{2mm}
\\
Then equation (\ref{1}) is uniformly exponentially stable.
\end{corollary}


\begin{corollary}\label{corollary3}
Assume that (a1)-(a3) hold, $t-g(t)\leq \sigma$,  
there exists $\lim_{t\rightarrow\infty} (t-h(t))=\tau$  and for some $t_0 \geq 0$ and $\alpha\in [0,1]$, for $t\geq t_0$, 
we have $a(t) \geq a_0>0$ and
\begin{equation}\label{12a}
\frac{\alpha (1-\|a\|_{[t_0,\infty)})}{e}< \tau \|b\|_{[t_0,\infty)}
<(1-\|a\|_{[t_0,\infty)})\left(1+\frac{\alpha}{e}\right)-\frac{\sigma \|a\|_{[t_0,\infty)} 
\|b\|_{[t_0,\infty)} (1-a_0)}{(1-\|a\|_{[t_0,\infty)})^2 }.
\end{equation}
Then equation (\ref{1}) is uniformly exponentially stable.
\end{corollary}
\begin{proof}
Since  $\lim\limits_{t\rightarrow\infty} (t-h(t))=\tau$, for any $\varepsilon>0$, in particular, for 
$$
\varepsilon < \min \left\{ \tau - \frac{\alpha (1-\|a\|_{[t_0,\infty)})}{e \|b\|_{[t_0,\infty)}},
\frac{1-\|a\|_{[t_0,\infty)}}{\|b\|_{[t_0,\infty)} } \left(1+\frac{\alpha}{e}\right)
-\frac{\sigma \|a\|_{[t_0,\infty)} 
 (1-a_0)}{(1-\|a\|_{[t_0,\infty)})^2 } 
\right\}
$$
there exists $t_1 \geq t_0$ such that
$\tau-\varepsilon \leq t-h(t)\leq \tau+\varepsilon$,  $t\geq t_1$.
Evidently for $t\geq t_1$ the conditions of Theorem~\ref{theorem1} hold:
$\alpha\tau_0 < \tau-\varepsilon$  and  
$\displaystyle (\tau+\varepsilon) \|b\|_{[t_0,\infty)}+\frac{\sigma \|a\|_{[t_0,\infty)} 
\|b\|_{[t_0,\infty)} (1-a_0)}{(1-\|a\|_{[t_0,\infty)})^2 }< \left( 1-\|a\|_{[t_0,\infty)} \right) \left(1+\frac{\alpha}{e}\right)$.
\end{proof}

Consider now two partial cases of equation (\ref{1}), one with constant coefficients
\begin{equation}\label{13}
(x(t)-ax(g(t)))'=-b(t)x(h(t)),
\end{equation}
where $a$ is a positive constant, and another with a non-delayed term
\begin{equation}\label{14}
(x(t)-a(t)x(g(t)))'=-b(t)x(t)
\end{equation}
and assume that the suitable parts of conditions (a1), (a2), (a4) hold for these equations.

\begin{corollary}\label{corollary1}
Assume that for some $t_0 \geq 0$ and  $\alpha\in [0,1]$, for $t\geq t_0$ we have
 $\alpha\leq \frac{\delta e\|b\|_{[t_0,\infty)}}{1-a}$  and  
\vspace{2mm}
\\
$\displaystyle \tau \|b\|_{[t_0,\infty)}+\frac{\sigma a
\|b\|_{[t_0,\infty)}}{1-a }<(1-a)\left(1+\frac{\alpha}{e}\right).
$
Then equation (\ref{13}) is uniformly exponentially stable.
\end{corollary}

\begin{corollary}\label{corollary2}
If $a(t) \geq a_0>0$, 
$\displaystyle 
\frac{\sigma \|a\|_{[t_0,\infty)}\|b\|_{[t_0,\infty)} (1-a_0)}{(1-\|a\|_{[t_0,\infty)})^3}<1
$
~ then equation (\ref{14}) is uniformly exponentially stable.
\end{corollary}

In Theorem~\ref{theorem1} and its corollaries we assume that $a(t)\geq 0$. In the next 
theorem we remove this restriction.

For any $u \in {\mathbb R}$ denote $u^+=\max\{u, 0\}, u^-=\max\{-a, 0\}$, hence $u=u^+-u^-$.

Define
\begin{equation}
\label{bar_tau}
\bar{\tau} = \frac{1- \| a^+ \|_{[t_0,+\infty)}}{e \| b \|_{[t_0,+\infty)}},~\displaystyle 
(S^+y)(t)=\left\{\begin{array}{ll}
a^+(t)y(g(t)),& g(t)\geq t_0,\\
0,& g(t)<t_0.\\
\end{array}\right. 
\end{equation}

\begin{guess}\label{theorem2}
Assume that (a1),(a2),(a4) hold and there are $t_0 \geq 0$ and $\alpha\in [0,1]$ such that for $t\geq t_0$,
 $\alpha\bar{\tau}\leq \delta$  and 
$$
\tau \|b\|_{[t_0,\infty)}+\frac{\sigma \|a^+\|_{[t_0,\infty)} 
\|b\|_{[t_0,\infty)} }{(1-\|a^+\|_{[t_0,\infty)})^2 }+\frac{\|a^-\|_{[t_0,\infty)}\|b\|_{[t_0,\infty)}}{1-\|a^+\|_{[t_0,\infty)}}
<1-\|a\|_{[t_0,\infty)}+\frac{\alpha(1-\|a^+\|_{[t_0,\infty)})}{e}.
$$
Then equation (\ref{1}) is uniformly exponentially stable.
\end{guess}

\begin{proof}
We follow the proof of Theorem~\ref{theorem1}.
Consider initial value problem (\ref{10})
with  $\|f\|_{[t_0,+\infty)}< +\infty$, $f(t)=0$ for $t \in [t_0, t_0+\tau]$  and let $I=[t_0,t_1]$ for some $t_1>t_0+\tau$. 
By (\ref{star})
the solution  $x$ of (\ref{10}) and the derivative of the function $y(t)=x(t)-a(t)x(g(t))=x(t)-a^+(t)x(g(t))+a^-(t)x(g(t))$ satisfy
$y'(t)=-b(t)x(h(t))+f(t)$, $x=(E-S)^{-1}y=(E-S^+)^{-1}y-(E-S^+)^{-1}(a^-(\cdot)x(g(\cdot)))$ and
\begin{equation}
\label{star1a}
|x|_I\leq \frac{1}{1-\|a\|_{[t_0,+\infty)}}|y|_I,~~
|y^{\prime}|_I\leq \frac{\|b\|_{[t_0,+\infty)}}{1-\|a\|_{[t_0,+\infty)}}|y|_I+\|f\|_{[t_0,+\infty)}.
\end{equation}
By Lemma~\ref{lemmaS}, (\ref{10})  
is equivalent to the equation with an infinite number of delays
\begin{equation}\label{11b}
y'(t)=
 -  b(t)\sum_{j=0}^{\infty} \prod_{k=0}^{j-1} a^+\left(h( g^{[k]}(t)) \right) 
y \left( h(g^{[j]}(t) ) \right)+b(t)(E-S^+)^{-1} \left( a^-(t)x(g(t)) \right) +f(t). 
\end{equation}
Denote 
\begin{equation}
\label{barB}
\bar{B}(t)= b(t)\sum_{j=0}^{\infty} \prod_{k=0}^{j-1} a^+\left( h(g^{[k]}(t)) \right),
\end{equation}
where we can assume that $a(t)=0$ for $t<t_0$.
Consider the delay equation
\begin{equation}\label{12b}
x'(t)+\bar{B}(t)x(t-\alpha\bar{\tau})=0.
\end{equation}
where $\bar{\tau}$  is defined in \eqref{bar_tau}.  

Using the bounds for $a$ and $b$, we obtain
$ \displaystyle b_0\leq \bar{B}(t)\leq \frac{\|b\|_{[t_0,\infty)}}{1-\|a^+\|_{[t_0,\infty)}}.$
Equation (\ref{11b}) can be rewritten in the form 
$$
y'(t)+\bar{B}(t)y(t-\alpha \tau_0)= 
b(t)\sum_{j=0}^{\infty} \prod_{k=0}^{j-1} a^+\left(h( g^{[k]}(t)) \right) \!\!\!\!
\int\limits_{h(g^{[j]}(t))}^{t-\alpha \tau_0} \!\!\!\! y'(\xi)d\xi+b(t)(E-S^+)^{-1} \left( a^-(t)x(g(t)) \right) + f(t).
$$

Since $\bar{B}(t)\geq b_0$ and $\alpha\bar{\tau} \bar{B}(t)\leq \frac{1}{e}$, by Lemma~\ref{lemma5} 
equation (\ref{12b}) is exponentially stable, and its fundamental function is positive:  $X_1(t,s)>0$, 
$t\geq s\geq t_0$.

By the same calculations as in Theorem~\ref{theorem1} we have 
$$\displaystyle
|y|_I \leq \left[\left(\tau - \alpha \bar{\tau} +\frac{\sigma \|a^+\|_{[t_0,\infty)}}{(1-\|a^+\|_{[t_0,\infty)})^2}\right)
\frac{\|b\|_{[t_0,\infty)}}{1-\|a\|_{[t_0,\infty)}}
+\frac{\|b\|_{[t_0,\infty)}\|a^-\|_{[t_0,\infty)}}{1-\|a^+\|_{[t_0,\infty)}}\frac{1}{1-\|a\|_{[t_0,\infty)}}\right]
|y|_I+M_2,
$$
where the constant $M_2$ does not dependent on $I$. 

The conditions of the theorem imply
$$\left(\tau - \alpha \bar{\tau} +\frac{\sigma \|a^+\|_{[t_0,\infty)}}{(1-\|a^+\|_{[t_0,\infty)})^2}\right)
\frac{\|b\|_{[t_0,\infty)}}{1-\|a\|_{[t_0,\infty)}}
+\frac{\|b\|_{[t_0,\infty)}\|a^-\|_{[t_0,\infty)}}{1-\|a^+\|_{[t_0,\infty)}}\frac{1}{1-\|a\|_{[t_0,\infty)}}<1.
$$

Hence $|y(t)|\leq M$ for $t \geq t_0$, for some constant $M$ which does not depend on the interval $I$.
Therefore $x$ is a bounded function 
on $[t_0,\infty)$. By Lemma \ref{lemma3}, equation (\ref{1}) is uniformly exponentially stable.
\end{proof}

\begin{remark}
If, in addition to the assumptions of Theorem~\ref{theorem2},
$\displaystyle  \|a\|_{[t_0,\infty)} = \|a^+\|_{[t_0,\infty)}$, which is equivalent to
$\displaystyle  \sup_{t \geq t_0} a(t) \geq \sup_{t \geq t_0} (-a(t))$, and
$$\displaystyle \tau \|b\|_{[t_0,\infty)}+\frac{\sigma \|a \|_{[t_0,\infty)} 
\|b\|_{[t_0,\infty)} }{(1-\|a \|_{[t_0,\infty)})^2 }+\frac{\|a^-\|_{[t_0,\infty)}\|b\|_{[t_0,\infty)}}{1-\|a \|_{[t_0,\infty)}}
< \left( 1-\|a\|_{[t_0,\infty)} \right)\left(1+\frac{\alpha}{e}\right)
$$
then equation (\ref{1}) is uniformly exponentially stable.
\end{remark}

Assuming $\alpha=1$ and $\alpha=0$ in Theorem~\ref{theorem2}, we get the following tests.

\begin{corollary}\label{corollary5}
Assume  that conditions (a1),(a2),(a4) are satisfied, and for $t\geq t_0$ at least one of the following conditions holds:
\\
a) $\bar{\tau}<\delta$ and \\
$\displaystyle \tau \|b\|_{[t_0,\infty)}+\frac{\sigma \|a^+\|_{[t_0,\infty)} 
\|b\|_{[t_0,\infty)} }{(1-\|a^+\|_{[t_0,\infty)})^2}+\frac{\|a^-\|_{[t_0,\infty)}\|b\|_{[t_0,\infty)}}{1-\|a^+\|_{[t_0,\infty)}}
<1-\|a\|_{[t_0,\infty)}+\frac{1-\|a^+\|_{[t_0,\infty)}}{e}.
$
\vspace{2mm}
\\
b)
$\displaystyle \tau \|b\|_{[t_0,\infty)}+\frac{\sigma \|a^+\|_{[t_0,\infty)} 
\|b\|_{[t_0,\infty)} }{(1-\|a^+\|_{[t_0,\infty)})^2 }+\frac{\|a^-\|_{[t_0,\infty)}\|b\|_{[t_0,\infty)}}{1-\|a^+\|_{[t_0,\infty)}}
<1-\|a\|_{[t_0,\infty)}.
$
\vspace{2mm}
\\
Then equation (\ref{1}) is uniformly exponentially stable.
\end{corollary}

In the following theorem, we
do not assume boundedness of delays.

\begin{guess}\label{theorem3} 
Assume that (a1)-(a3) hold, $\int_{t_0}^{\infty} b(s) ds=\infty$, $b(t)\neq 0$ almost everywhere, 
and there exist $t_0\geq 0, \alpha>0$, $\tilde{\delta}\geq 0$, $\tilde{\tau}>0$ and $\tilde{\sigma} >0$ such that for $t \geq t_0$,
\begin{equation}\label{14a}
\tilde{\delta}\leq \int_{h(t)}^t b(\xi)d\xi\leq \tilde{\tau},~ \int_{g(t)}^t b(\xi)d\xi\leq \tilde{\sigma},~ 
\alpha\tilde{\tau}_0 \leq \tilde{\delta}, ~~ \tilde{\tau}_0:=\frac{1-\|a \|_{[t_0,\infty)}}{e}
\end{equation}
and  
\begin{equation}\label{13a}
\displaystyle \tilde{\tau} +\frac{\tilde{\sigma} \|a\|_{[t_0,\infty)} 
(1-a_0)}{(1-\|a\|_{[t_0,\infty)})^2 }<(1-\|a\|_{[t_0,\infty)})\left(1+\frac{\alpha}{e}\right).
\end{equation}
Then equation (\ref{1}) is asymptotically  stable.
\end{guess}
\begin{proof}
We follow the scheme of the proof for Theorem~\ref{theorem1}. 
Denote $y(t)=x(t)-a(t)x(g(t))$. Then  (\ref{1}) is equivalent to the equation with an infinite number of delays:
\begin{equation}\label{14b}
y'(t)=
 -  b(t)\sum_{j=0}^{\infty} \prod_{k=0}^{j-1} a\left(h( g^{[k]}(t)) \right) 
y \left( h(g^{[j]}(t) ) \right), 
\end{equation}
where 
$$
g^{[0]}(t)=t, ~g^{[1]}(t)=g(t),~ g^{[k]}(t) = g\left( g^{[k-1]}(t) \right),~ ~~k\in {\mathbb N}. 
$$
Let ${\displaystyle s=p(t):=\int_{t_0}^t b(\tau)d\tau,~s\geq 0,  z(s)=y(t)}$,
where $p(t)$ is a strictly increasing function.
Then we introduce $\tilde{h}(s)$ and $\tilde{g}(s)$
as follows:
$$
y(h(t))=z(\tilde{h}(s)), ~\tilde{h}(s)\leq s, ~ \tilde{h}(s)=\int_{t_0}^{h(t)} 
b(\tau)d\tau, ~ s-\tilde{h}(s)=\int_{h(t)} ^t b(\tau)d\tau, 
$$ 
$$
\tilde{g}(s)=\int_{t_0}^{g(t)} b(\tau)d\tau,~ s-\tilde{g}(s)=\int_{g(t)} ^t b(\tau)d\tau, ~\tilde{g}(s)\leq s,
$$
$$
\dot{y}(t)=b(t)\dot{z}(s),~~ \dot{y}(g(t))=b(g(t))\dot{z}(\tilde{g}(s)).
$$
Equation (\ref{14b}) can be rewritten in the form
\begin{equation}\label{14c}
z'(s)=
 -  \sum_{j=0}^{\infty} \prod_{k=0}^{j-1} a\left(\tilde{h}( \tilde{g}^{[k]}(s)) \right) 
z \left( \tilde{h}(\tilde{g}^{[j]}(s) ) \right). 
\end{equation}
Consider the initial value problem 
\begin{equation}\label{14d}
z'(s)=
 -  \sum_{j=0}^{\infty} \prod_{k=0}^{j-1} a\left(\tilde{h}( \tilde{g}^{[k]}(s)) \right)
z \left( \tilde{h}(\tilde{g}^{[j]}(s) ) \right)+f(s), ~~z(s)=0,~ s\leq 0
\end{equation}
with  $\|f\|_{[0,+\infty)}< +\infty$, $f(s)=0$ for $s \in [0, \tilde{\tau}]$. 

Problem (\ref{14d}) has a form of problem (\ref{11}) with $b(t)\equiv 1$, where condition (\ref{13a}) corresponds to (\ref{1a}).
As in the proof of Theorem~\ref{theorem1}, one can show that the solution $z$ of (\ref{14d}) is a bounded on $[t_0,\infty)$
function. Then equation (\ref{14c}) is exponentially stable. Hence for any solution $z$ of (\ref{14c}), $\lim\limits_{s\rightarrow\infty} z(s)$=0.
If $y$ is a solution of (\ref{14b}) then $\lim\limits_{t\rightarrow\infty} y(t)=\lim\limits_{s\rightarrow\infty} z(s)=0$. 
If $x$ is a solution of (\ref{11}) then $x=(E-S)^{-1}y$, therefore $\lim\limits_{t\rightarrow\infty} x(t)=0$ as well.
Finally, equation (\ref{1}) is asymptotically  stable.
\end{proof}

\section{Examples and Conclusion}

First, we compare the results of the present paper with Propositions~\ref{proposition2} and \ref{proposition2a}. 

\begin{example}
\label{example3}
Consider an equation with constant delays and variable coefficients
\begin{equation}
\label{ex2eq1}
\left( x(t)-(0.498+0.001\cos t)x(t-\pi) \right)^{\prime}=-r(0.9+0.1 \sin t)x(t-\pi),
\end{equation}
which allows comparison to known results. 
Denote $\overline{r}:= \displaystyle \limsup_{t \to\infty} r \int_{t-\pi}^t (0.9+0.1\sin s)~ds=r(0.9\pi+2)$.
Here $a_0=0.497$, $A_0=0.499$.
Corollary~\ref{main} b) 
implies exponential stability for $0.059<\overline{r}<0.109$,
while part a) gives $0<\overline{r}<0.0797$, which overall gives  exponential stability for $\overline{r}<0.109$.
Since 
$\sqrt{2 \left( 1-2A_0 \right) } \approx 0.0632$, the conditions of Proposition~\ref{proposition2a} 
hold for $0<\overline{r}<0.0632$. The assumptions of Proposition~\ref{proposition2} are satisfied for $0<\overline{r}<0.002$,
so Corollary~\ref{main} of Theorem~\ref{theorem1} leads to a sharper result. 
\end{example}

Next, we illustrate the results of the paper with examples for which all previous tests fail.
In the following two examples, we outline the role of $\alpha$ in Theorem~\ref{theorem1}.

\begin{example}
\label{example1}
Consider the equation
\begin{equation}\label{15}
\left( x(t)-0.6x(t-0.2 |\sin t|) \right)'=-x(t-0.14).
\end{equation}
Here $a(t)=0.6$, $\|b\|_{[t_0,\infty)} = 1$, $\tau=\delta=0.14$, $\sigma=0.2$.
Condition \eqref{12a} in Corollary~\ref{corollary3} for equation \eqref{15}
has the form
\begin{equation}\label{16}
\frac{0.4\alpha}{e}\leq 0.14<0.1+\frac{0.4\alpha}{e}.
\end{equation}
For $\alpha\in (0.1e,0.35e] \approx (0.272,0.951]$ condition \eqref{16} holds, in particular for $\alpha=0.5$. 
Hence by  Corollary~\ref{corollary3} equation 
\eqref{15} is uniformly exponentially stable. 
Note that for $\alpha=0$ and for $\alpha=1$ 
condition \eqref{16} fails. 

Propositions~\ref{proposition2} and \ref{proposition2a} are not applicable, as the delay of the neutral term is variable.
Moreover, the coefficient $a(t)$ in the neutral part in Propositions~\ref{proposition2} and \ref{proposition2a} must be less
than $0.5$, while in \eqref{16} $a(t)=0.6>0.5$. 
\end{example}


\begin{example}
\label{ex2new}
Consider an equation with variable coefficients and delays
\begin{equation}
\label{ex1eq1}
\left( x(t)-(0.5+0.1\cos t)x(g(t)) \right)^{\prime}=-r(0.9+0.1 \sin t)x(t-1), ~0.9 \leq t-g(t)\leq 1. 
\end{equation}
We have $\tau=\sigma=\delta=1$, $a_0=0.4$, $\|a\|_{[t_0,\infty)}=0.6$, $\|b\|_{[t_0,\infty)}=r$.
Then the assumptions of Theorem~\ref{theorem1} are equivalent to
$$
\alpha \frac{0.4}{er} \leq 1, ~~ r+r \frac{0.6^2}{0.4^2}< 0.4 \left( 1+ \frac{\alpha}{e} \right) ~~
\Leftrightarrow ~
  \frac{0.4\alpha}{e} \leq r \leq 0.4 \frac{4}{13}  \left( 1+ \frac{\alpha}{e} \right). $$
This corresponds to the values between two lines in $(\alpha,r)$ graph on Fig.~\ref{figure1}. The highest allowed value
of $r$ is achieved for $\alpha=1$ and corresponds to Part a) of Corollary~\ref{main}, $r < r_0 \approx 0.168$.
To the best of our knowledge, other known tests are not applicable to (\ref{ex1eq1}). 

\begin{figure}[ht]
\centering
\includegraphics[scale=0.4]{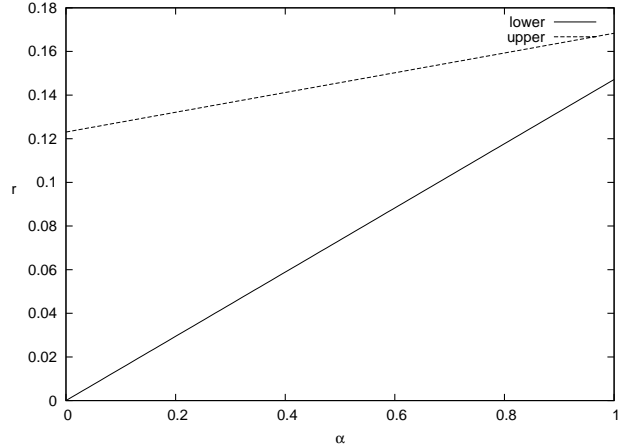} 
\caption{The lower and the upper bounds for $r$ deduced from Theorem~\protect{\ref{theorem1}}
for which equation (\protect{\ref{ex1eq1}}) is asymptotically stable for any $\alpha \in [0,1]$. } 
\label{figure1}
\end{figure}

\end{example}

\begin{example}
\label{ex3}
Consider an equation with an oscillating coefficient in the neutral term
\begin{equation}\label{20}
\left( x(t)-0.6 (\sin t) x(t-0.2|\cos t|) \right)'=-0.1(0.5+|\sin t|)x(t-0.5).
\end{equation}
Here 
$$
\|a\|_{[t_0,\infty)}=\|a^+\|_{[t_0,\infty)}=\|a^-\|_{[t_0,\infty)}=0.6, \|b\|_{[t_0,\infty)}=0.15, \tau=0.5, \sigma=0.2.
$$
The inequality in Theorem~\ref{theorem2} has the form $\displaystyle 0.4125< 0.4+0.147 \alpha$ which holds
for $\alpha>0.085$. Note that, as $\delta=\tau=0.5$ and $\bar{\tau}\approx 0.98$ in \eqref{bar_tau}, any $\alpha <0.5$ satisfies
$\alpha \bar{\tau}< \delta$. For $\alpha=0.45$ we get the correct inequality $0.4125<0.4147$,
hence Equation \eqref{20} is uniformly exponentially stable.

Propositions \ref{proposition2}, \ref{proposition2a} and Theorem~\ref{theorem1} 
fail for Equation \eqref{20}.
\end{example}

Finally, we apply Theorem~\ref{theorem3} to an equation with unbounded delays.

\begin{example}
\label{ex5}
Consider a pantograph-type neutral equation
\begin{equation}\label{22}
\left( x(t)- 0.55 x \left(\frac{t}{3} \right) \right)'=- \frac{1}{4t} x \left(\frac{t}{2} \right),~~t\geq t_0>0.
\end{equation}
Then in Theorem~\ref{theorem3} we have $\tilde{\sigma}=0.25 \ln 3$, $\tilde{\delta}=\tilde{\tau}=0.25 \ln 2$,
$\tau_0 \approx 0.1655<\tilde{\delta} \approx 0.173$, so any $\alpha$ can be used in \eqref{14a}.
Any $\alpha \geq 0.36$ implies asymptotic stability; in particular, for $\alpha=1$ 
inequality \eqref{13a} has the form $0.509<0.6$ and thus holds.
\end{example}

In the present paper, we considered neutral equation \eqref{1} in the most general framework:
all coefficients and delays are measurable functions. We also do not assume  
coefficient in  the neutral part $\|a(t)\| < \frac{1}{2}$, which is usually imposed,
see for example Propositions \ref{proposition2} and  \ref{proposition2a}. 
In four of the five examples, 
$a(t)$ can exceed 0.5.

The method used in the paper is based on the Bohl-Perron theorem and a transformation
of the given neutral equation to an equation with an infinite number of delays.
This scheme is applied to neutral equations in the Hale form for the first time,
and we illustrated its efficiency.

\section*{Acknowledgment}

The second author was partially supported by the NSERC research grant RGPIN-2015-05976.

\end{document}